\documentclass{article}
\usepackage{amssymb,amsfonts,amsmath,amsthm,amsopn,amstext,amscd,latexsym,xy,graphicx}
\usepackage{hyperref}
\usepackage{stmaryrd}

\input xy
\xyoption{all}

\newtheorem{theorem}{Theorem}[section]

\newtheorem{proposition}[theorem]{Proposition}
\newtheorem{corollary}[theorem]{Corollary}
\newtheorem{definition}[theorem]{Definition}

\theoremstyle{remark}

\newtheorem*{remark}{Remark}

\renewcommand{\)}{\textup{)}}

\DeclareMathOperator{\GL}{GL}
\DeclareMathOperator{\SL}{SL}

\DeclareMathOperator{\Frob}{Frob}
\DeclareMathOperator{\Sym}{Sym}

\DeclareMathOperator{\Gal}{Gal}

\DeclareMathOperator{\tr}{tr}

\DeclareMathOperator{\Lie}{Lie}

\newcommand{\cE}{{\mathcal E}}
\newcommand{\cF}{{\mathcal F}}

\newcommand{\frg}{{\mathfrak g}}

\newcommand{\bbA}{{\mathbb A}}

\newcommand{\bbC}{{\mathbb C}}

\newcommand{\bbF}{{\mathbb F}}
\newcommand{\bbG}{{\mathbb G}}

\newcommand{\bbP}{{\mathbb P}}
\newcommand{\bbQ}{{\mathbb Q}}
\newcommand{\bbR}{{\mathbb R}}

\newcommand{\bbZ}{{\mathbb Z}}

\newcommand{\PGL}{{\mathrm{PGL}}}

\title{Equidistribution of Frobenius eigenvalues}
\author{Jack A. Thorne}

\begin{document}

\maketitle
\begin{abstract}
We study the problem of variation of Frobenius eigenvalues on the cohomology of families of local systems of algebraic curves over finite fields.
\end{abstract}

\tableofcontents

\section{Introduction}

Our starting point in this article is the following result of Serre \cite{Ser97}.
\begin{theorem}\label{thm_intro_serre_theorem}
Fix an integer $N \geq 1$ and a prime $p \nmid N$, and consider for each even integer $k \geq 2$ the set $\mathbf{y}_k \subset [ -2 p^{(k-1)/2}, 2 p^{(k-1)/2}]$ of eigenvalues of the operator $T_p$, with multiplicity, on the space $S_k(\Gamma_0(N), \bbC)$ of weight $k$ cusp forms. Let $\mathbf{x}_k = p^{-(k-1)/2} \cdot \mathbf{y}_k \subset [-2, 2]$. Then the sets $\mathbf{x}_k$ become equidistributed as $k \to \infty$ with respect to the measure
\[ \mu_p(f) = \int_{-2}^2 f(x) \frac{p+1}{\pi} \frac{\sqrt{1 - x^2/4}}{(\sqrt{p} + 1/\sqrt{p})^2 - x^2} dx, \, f \in C([0, 1], \bbR).\]
\end{theorem}
The proof of Serre's theorem involves understanding for each $n \geq 1$ the asymptotic (as $k \to \infty$) behaviour of the renormalized traces 
\[ \frac{\tr T_p^n}{p^{-n(k-1)/2} \cdot \dim_\bbC S_k(\Gamma_0(N), \bbC)} \] 
using the Eichler--Selberg trace formula.

As is well-known since Deligne \cite{Del71}, the eigenvalues of the operator $T_p$ can also be studied using algebraic geometry. Suppose that $N \geq 5$, let $Y_1(N)$ denote the modular curve of level $N$ over $\bbF_p$, and let $j : Y_1(N) \hookrightarrow X_1(N)$ denote its usual compactification. Then $X_1(N)$ is a smooth, projective, geometrically connected curve over $\bbF_p$. For each $k \geq 2$ and prime $l \neq p$, there is a lisse $\overline{\bbQ}_l$-sheaf on $Y_1(N)$ given by $\cF = R^1 \pi_\ast \overline{\bbQ}_l$, where $\pi : \cE^\text{univ} \to Y_1(N)$ is the universal elliptic curve. We then have an identity:
\begin{equation}\label{eqn_intro_hecke_trace} \tr (T_p \mid S_k(\Gamma_1(N), \bbC)) = \tr( \Frob_p \mid H^1(X_1(N)_{\overline{\bbF}_p}, j_\ast \Sym^{k-2} \cF )), \end{equation}
and the equidistribution of the (re-normalized) eigenvalues of $T_p$ is essentially equivalent to the equidistribution of the (re-normalized) eigenvalues of $\Frob_p$ as $k \to \infty$. (We remark that the left hand side of (\ref{eqn_intro_hecke_trace}) is an element of $\bbC$, while the right hand side lies in $\bbQ_l$. The identity has a sense because both sides are in fact rational integers.)  It then seems natural to ask if one can generalize Theorem \ref{thm_intro_serre_theorem} to a statement about equidistribution of Frobenius eigenvalues for local systems on arbitrary curves over finite fields.

In this paper, we answer this question in the affirmative for some natural families of local systems. We consider a smooth, projective, geometrically connected curve $X_0$ over $\bbF_q$, together with a Zariski open subset $U_0 \subset X_0$ and a lisse $\overline{\bbQ}_l$-sheaf $\cF_0$ on $U_0$. We assume that the following conditions hold:
\begin{itemize}
\item The pullback $\cF$ of $\cF_0$ to $U = U_{0, \overline{\bbF}_q}$ is irreducible. Equivalently, the representation $\rho : \pi_1(U, \overline{\eta}) \to \GL(\cF_{\overline{\eta}})$ is irreducible. (We write $\overline{\eta}$ for a choice of geometric generic point of $U_0$.)
\item The geometric monodromy group $G$ (i.e.\ the Zariski closure of the image of $\rho$) and the arithmetic monodromy group $G_0$ (i.e.\ the Zariski closure of the image of $\rho_0 : \pi_1(U_0, \overline{\eta}) \to \GL(\cF_{0,\overline{\eta}})$) are equal, and connected. (By a result of Grothendieck, $G = G_0$ is then a semi-simple algebraic group.) 
\end{itemize}
Finally, we fix an isomorphism $\iota : \overline{\bbQ}_l \cong \bbC$. For each irreducible algebraic representation $\xi$ of the semi-simple group $G$, we can define a new lisse sheaf $\cF_{0, \xi}$ on $U_0$, which is associated to the representation $\xi \circ \rho_0$ of $\pi_1(U_0, \overline{\eta})$. The semi-simplicity of $G_0$ implies that the sheaf $\cF_{0, \xi}$ is $\iota$-pure of weight $0$.

It then follows from Deligne's proof of the Weil conjectures that for each eigenvalue $\gamma$ of $\Frob_q$ on $H^1(X, j_\ast \cF_\xi)$, we have $| \iota(\gamma) |^2 = q$. We can therefore consider the set $\mathbf{y}_\xi \subset S^1$ of $q^{-1/2}$-multiples of the eigenvalues $\iota(\gamma)$ of $\Frob_q$, taken with multiplicity, and ask whether they satisfy any kind of equidistribution property. We prove the following result in this direction.
\begin{theorem}[Theorem \ref{thm_equidistribution_general_case}]\label{thm_intro_equidistribution_theorem}
Suppose that $G \neq 1$, and let $\xi_1, \xi_2, \dots$ be a sequence of irreducible algebraic representations of $G$ of fixed central character $\omega$ such that $\xi_i \to \infty$ as $i \to \infty$. \(This means that the highest weight of $\xi_i$ tends to infinity `far from the walls'; see Definition \ref{defn_representations_tend_to_infinity} below.\) Suppose that the genus of $X$ is at least 2. Then there exists a measure $\nu$ on $S^1$ such that the sets $\mathbf{y}_{\xi_i}$ become $\nu$-equidistributed as $i \to \infty$.
\end{theorem}
In fact, we describe the measure $\nu$ explicitly by calculating its moments. If $X$ is a modular curve that carries a universal family of elliptic curves, then we essentially recover Serre's Theorem \ref{thm_intro_serre_theorem}. We discuss more general families of elliptic curves in \S \ref{sec_examples} below.

It is interesting to compare Theorem \ref{thm_intro_equidistribution_theorem} with Serre's \cite[Th\'eor\`eme 8]{Ser97}. Maintaining the analogy with modular forms, the theorem of \emph{loc. cit.} describes equidistribution of Frobenius eigenvalues in the level aspect (i.e.\ eigenvalues in $H^1(X_i, \overline{\bbQ}_l)$, as $X_{i, 0}$ ranges through a sequence of curves over over $\bbF_q$), while our theorem describes equidistribution of Frobenius eigenvalues in the weight aspect.

We now discuss the organization of this paper. In \S \ref{sec_equidistribution_on_the_circle}, we make precise our notion of equidistribution with respect to a measure. In \S \ref{sec_elements_of_compact_lie_groups}, we recall a result about the asymptotic behaviour of expressions like $\tr \xi(\gamma) / \dim \xi$, where $\gamma$ is an element of a compact Lie group $K$ and $\xi$ is a sequence of representations of $K$ that `tend to infinity'. This is the main technical input in our proof of Theorem \ref{thm_intro_equidistribution_theorem}. In \S \ref{sec_frobenius_eigenvalues}, we prove Theorem \ref{thm_intro_equidistribution_theorem}, essentially by direct calculation with the Lefschetz trace formula. Finally, in \S \ref{sec_examples} we discuss examples arising from families of elliptic curves, and in \S \ref{sec_examples_2} we discuss a family of examples arising from Deligne's Kloosterman sheaves.
\subsection{Funding}

This research was partially conducted during the period the author served as a Clay Research Fellow.

\subsection{Acknowledgments}

We thank the referee for several useful suggestions that improved the exposition in this paper.

\subsection{Notation} If $G$ is a compact Lie group, and $f, g$ are continuous complex-valued functions on $G$, then we define $\langle f, g \rangle_G = \int f \overline{g} \,d\mu_G$, where $d\mu_G$ is the probability Haar measure on $G$. 

\section{Equidistribution in $S^1$}\label{sec_equidistribution_on_the_circle}

Let $S^1$ denote the unit circle in $\bbC^\times$. Suppose given for each $i \geq 1$ a finite multiset $\mathbf{x}_i = \{ x_{i, j} \}$ of points of $S^1$. In this section, we define what it means for the sets $\mathbf{x}_i$ to become equidistributed in $S^1$ as $i \to \infty$. 

We write $C(S^1, \bbR)$ for the space of continuous, real-valued functions on $S^1$, endowed with its supremum norm. Similarly, we write $C(S^1, \bbC)$ for the space of continuous, complex-valued functions on $S^1$, endowed with its supremum norm. By definition, a measure $\nu$ on $S^1$ is a linear functional $\nu : C(S^1, \bbR) \to \bbR$ such that $\nu(f) \geq 0$ if $f \geq 0$ and $\nu(1) = 1$ (see \cite[Ch. III]{Bou04}). A measure $\nu$ being given, we extend its domain of definition to $C(S^1, \bbC)$ in the obvious manner.
\begin{definition}
We say that the sets $\mathbf{x}_i$ become $\nu$-equidistributed in $S^1$ as $i \to \infty$ if for all $f \in C(S^1, \bbC)$, we have
\[ \lim_{i \to \infty} \frac{ \sum_j f(x_{i, j}) }{\# \mathbf{x}_i} = \nu(f). \]
\end{definition}
\begin{proposition}\label{prop_existence_of_limiting_measure}
The following are equivalent:
\begin{enumerate}
\item For each $n \geq 1$, the limit 
\[ \omega_n = \lim_{i \to \infty} \frac{ \sum_j x_{i, j}^n }{\# \mathbf{x}_i} \]
exists.
\item There exists a measure $\nu$ such that the sets $\mathbf{x}_i$ become $\nu$-equidistributed as $i \to \infty$.
\end{enumerate}
In this case, the measure $\nu$ is uniquely determined.
\end{proposition}
\begin{proof}
The implication $(ii) \Rightarrow (i)$ is immediate from the definition. In the other direction, we let $V \subset C(S^1, \bbR)$ denote the subspace spanned by real-valued trigonometric polynomials (i.e.\ finite linear combinations of functions $z^k$, $k \in \bbZ$). By the Weierstrass approximation theorem, $V$ is a dense subspace of $C(S^1, \bbR)$. By assumption, we can define a positive linear functional $\nu : V \to \bbR$ by the formula
\[ \nu(f) = \lim_{i \to \infty} \frac{ \sum_j f(x_{i, j}) }{\# \mathbf{x}_i}. \]
By \cite[Ch. III, \S 1, No. 7, Proposition 9]{Bou04}, $\nu$ admits a unique extension to a positive linear functional $\nu : C(S^1, \bbR) \to \bbR$. This is the desired measure. 
\end{proof}
A pleasant situation occurs when the values $\omega_k$ satisfy an estimate of the form $| \omega_k | \leq C e^{- \alpha k}$ for some constants $C, \alpha > 0$. In this case the Fourier series $\sum_{k \in \bbZ} \omega_k z^{-k}$ converges uniformly on $S^1$ to a real analytic function $F \in C(S^1, \bbR)$, and the measure $\nu$ can be given by the formula
\[ \nu(f) = \int_{S^1} f F \,d \mu, \]
where $\mu$ is the usual Lebesgue probability measure on $S^1$.

We now consider an example. Let $q > 1$, and define
\begin{equation}\label{eqn_definition_of_plancherel_measure} F(z) = 1 + \frac{q-1}{2} \sum_{n \neq 0} q^{-|n|} z^{2n} = 1 + (q-1) \sum_{n > 0} q^{-n} \cos 2 n \theta, 
\end{equation}
for $z = e^{i \theta} \in S^1$. The measure $F d \mu$ is invariant under the substitution $z \mapsto z^{-1}$. We can identify the quotient of $S^1$ by this substitution with $\Omega = [-2, 2]$ (via the map $z \mapsto z + z^{-1}$), and the pushforward of the measure $F d \mu$ to $\Omega$ is exactly the measure denoted $\mu_q$ in \cite[\S 2]{Ser97}. If $q$ is a prime power, and $K$ is a non-archimedean local field with residue field of order $q$, then this pushforward can be identified with Plancherel measure for the unramified principal series of $\PGL_2(K)$ (see \emph{loc. cit.}). The function $F$ will appear again in \S \ref{sec_examples} below.

\section{Elements of compact Lie groups}\label{sec_elements_of_compact_lie_groups}

Let $\Omega$ be a field equipped with an isomorphism $\iota : \Omega \cong \bbC$. Let $G$ be a reductive algebraic group over $\Omega$ (in particular, connected). Let $K \subset G(\bbC)$ be a maximal compact subgroup.
\begin{proposition} Let $\mathrm{Rep}_K$ denote the category of continuous representations of $K$ on finite-dimensional $\bbC$-vector spaces. Let $\mathrm{Rep}_G$ denote the category of algebraic representations of $G$ on finite-dimensional $\Omega$-vector spaces.
\begin{enumerate}
\item The natural functor $\mathrm{Rep}_G \to \mathrm{Rep}_K$ given by restriction to $K \subset G(\bbC) \cong G(\Omega)$ is an equivalence of categories.
\item Let $g, h \in K$. Then $g, h$ are $K$-conjugate if and only if they are $G(\Omega)$-conjugate.
\item Let $g \in G(\Omega)$, and let $V$ be a faithful representation of $G$. Suppose that the eigenvalues of $\iota g$ in $\iota V$ have complex absolute value $1$. Let $g = g_s g_u$ be the Jordan decomposition of $g$. Then $g_s$ is $G(\Omega)$-conjugate to an element of $K$, unique up to $K$-conjugacy.
\end{enumerate}
\end{proposition}
Compare \cite[\S 2.2]{Del80}.
\begin{proof}
The first part of the proposition is well-known. It implies the second part, since elements of $K$ are semi-simple, and semi-simple conjugacy classes of a reductive algebraic group are separated by characters. For the third part, we observe that the given hypothesis implies that $\iota(g_s)$ is contained in a compact subgroup of $G(\bbC)$, hence in a $G(\bbC)$-conjugate of $K$. This implies that $g_s$ is $G(\Omega)$-conjugate to $K$. The uniqueness of its $K$-conjugacy class is exactly the second part of the proposition.
\end{proof}
\begin{definition}\label{defn_representations_tend_to_infinity}
Let $\xi_1, \xi_2, \dots$ be a sequence of irreducible representations of $G$. Fix a maximal torus and Borel subgroup $T \subset B \subset G$, and let $S \subset \Phi(G, T) \subset X^\ast(T)$ be the corresponding set of simple roots. Corresponding to each $\xi_i$ is a $B$-dominant weight $\lambda_i \in X^\ast(T)$ \(i.e.\ satisfying $\langle \lambda_i, \alpha^\vee \rangle \geq 0$ for each $\alpha \in S$\). We say that $\xi_i$ `tends to infinity' as $i \to \infty$, and write $\xi_i \to \infty$, if for each $\alpha \in S$, we have $\langle \lambda_i, \alpha^\vee \rangle \to \infty$ as $i \to \infty$.
\end{definition}
It is clear that this definition is independent of the choice of $T$ and $B$. 
\begin{proposition}\label{prop_chenevier_clozel}
Let $\xi_1, \xi_2, \dots$ be a sequence of irreducible representations of $G$ such that $\xi_i \to \infty$ as $i \to \infty$. Let $g \in G(\Omega)$ be an element with semi-simple part $g_s$ conjugate into $K$. Suppose that $g_s$ does not lie in the center $Z_G$ of $G$. Then we have
\[ \lim_{i \to \infty} \frac{\iota\tr \xi_i(g)}{\dim_\Omega \xi_i} = 0. \]
\end{proposition}
\begin{proof}
We have $\tr \xi_i(g) = \tr \xi_i(g_s)$. The proposition is then exactly \cite[Corollaire 1.12]{Che09}. 
\end{proof}
\begin{proposition}\label{prop_jacobson_morozov}
Let $\xi_1, \xi_2, \dots$ be a sequence of irreducible representations of $G$ such that $\xi_i \to \infty$ as $i \to \infty$. Let $N \in \frg = \Lie G$ be a non-zero nilpotent element. Then we have
\[ \lim_{i \to \infty} \frac{\dim_\Omega \ker d\xi_i(N)}{\dim_\Omega \xi_i} = 0. \]
\end{proposition}
\begin{proof}
By the Jacobson--Morozov theorem, we can find a homomorphism $\varphi : \SL_2 \to G$ such that $d \varphi( E ) = N$ (where $F, H, E$ are the standard generators of the Lie algebra of the algebraic group $\SL_2$). After conjugation in $G(\Omega)$ (which does not affect the result), we can assume that $\varphi$ restricts to a homomorphism $\varphi : \mathrm{SU}_2(\bbR) \to K$.

If $m \geq 1$, let $V_m$ denote the irreducible representation of $\mathrm{SU}_2(\bbR)$ of dimension $m$. Decomposing
\[ \xi_i|_{\mathrm{SU}_2(\bbR)} \cong \oplus_m n_{i, m} V_m, \]
we calculate
\[ \dim_\Omega \ker d \xi_i(N) = \sum_m n_{i, m} \text{ and }\frac{\dim_\Omega \ker d\xi_i(N)}{\dim_\Omega \xi_i} = \frac{\sum_m n_{i, m}}{\sum_m n_{i, m} m}. \]
Let $T \subset \mathrm{SU}_2(\bbR)$ denote the standard diagonal maximal torus, which we identify with $S^1$. Then the restriction of the character of $V_m$ to $T$ is equal to $f_m(z) = z^{m-1} + z^{m-3} + \dots + z^{1-m}$. The restriction of the character of $\iota \xi_i$ to $T$ is equal to $\sum_m n_{i, m} f_m$. For each $i \geq 1$, we therefore have
\[ \langle 1 + z, \sum_m n_{i, m} f_m \rangle_{S^1} = \sum_{m} n_{i, m}. \]
Let $g_i = (\sum_m n_{i, m} f_m)/\dim_\Omega \xi_i$. For each $z \in S^1$, we have
\[ |g_i(z)| \leq \frac{\sum_m n_{i, m} |f_m(z)|}{\dim_\Omega \xi_i} \leq \frac{\sum_m n_{i, m} m}{\dim_\Omega \xi_i} = 1. \]
On the other hand, it follows from Proposition \ref{prop_chenevier_clozel} that $g_i(z) \to 0$ pointwise almost everywhere as $i \to \infty$. Indeed, we have $g_i(z) \to 0$ for any $z \in T$ such that $\varphi(z) \not\in Z_G$. By Lebesgue's dominated convergence theorem, we conclude that
\[ \lim_{i \to \infty} \langle 1 + z, g_i \rangle_{S^1} = \lim_{i \to \infty} \frac{\sum_{m} n_{i, m}}{\dim_\Omega \xi_i} = 0. \]
This is the desired result.
\end{proof}

\section{Frobenius eigenvalues}\label{sec_frobenius_eigenvalues}

We now come to our main problem. We use the formalism of lisse $\overline{\bbQ}_l$-sheaves, as discussed in \cite[\S 1]{Del80}.

Let $X_0$ be a smooth, projective, geometrically connected curve over $\bbF_q$. We fix a Zariski open subset $j_0 : U_0 \hookrightarrow X_0$, and set $S_0 = X_0 - U_0$. We write $k$ for a fixed choice of algebraic closure of $\bbF_q$. If $A_0$ is some kind of object defined over $\bbF_q$ (for example, one of the schemes $X_0$, $U_0$, or $S_0$, or an \'etale sheaf on one of these spaces) then we write $A$ for its base extension to $k$. Thus, for example, $X$ is a smooth, projective, connected curve over $k$. Let $\eta$ be the generic point of $U_0$, and let $\overline{\eta}$ be a geometric point of $U$ above $\eta$. Then there is a short exact sequence of \'etale fundamental groups
\[ \xymatrix@1{ 1 \ar[r] & \pi_1(U, \overline{\eta}) \ar[r] & \pi_1(U_0, \overline{\eta}) \ar[r]^-{\text{deg}} & \widehat{\bbZ} \ar[r] & 1, } \]
where the map $\deg$ takes Frobenius elements $\Frob_{q, x}$ (associated to points $x \in U_0(\bbF_q)$ to the canonical generator $1 \in \widehat{\bbZ}$. If $\cF_0$ is a lisse $\overline{\bbQ}_l$-sheaf on $U_0$, then its stalk $\cF_{0, \overline{\eta}}$ is a finite-dimensional $\overline{\bbQ}_l$-vector space which receives a continuous action of the profinite group $\pi_1(U_0, \overline{\eta})$.

Let $\cF_0$ be a lisse $\overline{\bbQ}_l$-sheaf on $U_0$ and $\cF$ its pullback to $U$. We suppose that $\cF_0$ satisfies the following conditions:
\begin{itemize}
\item The sheaf $\cF$ is simple, and the Zariski closure $G$ of the image of the map $\rho : \pi_1(U, \overline{\eta}) \to \GL(\cF_{\overline{\eta}})$ is connected.
\item The image of the map $\rho_0 : \pi_1(U_0, \overline{\eta}) \to \GL(\cF_{\overline{\eta}})$ is contained in $G$. Thus the arithmetic monodromy group of $\cF_0$ is equal to the geometric monodromy group.
\end{itemize}
Then $G$ is a semisimple algebraic group, by \cite[Corollaire 1.3.9]{Del80}. If $\xi$ is an irreducible algebraic representation of $G$, then we write $\cF_{\xi, 0}$ for the lisse $\overline{\bbQ}_l$-sheaf on $U_0$ which is associated to the composite homomorphism 
\[ \pi_1(U_0, \overline{\eta}) \to G(\overline{\bbQ}_l) \to \GL(\xi). \]
The Frobenius endomorphism $\Frob_q$ acts on the groups $H^1(X, j_\ast \cF_\xi)$, and our objective is to study the eigenvalues of $\Frob_q$ as $\xi$ varies. More precisely, let us fix an isomorphism $\iota : \overline{\bbQ}_l \cong \bbC$. The sheaves $\cF_{0, \xi}$ are punctually $\iota$-pure of weight $0$ (because the arithmetic monodromy group is connected semi-simple), and one knows (\cite[Th\'eor\`eme 3.2.3]{Del80}) that the group $H^1(X, j_\ast \cF_\xi)$ is $\iota$-pure of weight $1$; in other words, for each eigenvalue $\alpha$ of $\Frob_q$ on this space, we have $|\iota(\alpha)|^2 = q$. It follows that the eigenvalues of $\Frob_q \otimes q^{-1/2}$ on $H^1(X, j_\ast \cF_\xi) \otimes_{\overline{\bbQ}_l, \iota} \bbC$ all lie on the unit circle, and we wish to study the distribution of these subsets of $S^1$, taken with multiplicity, as $\xi$ varies.

In order to control the dimensions of the groups $H^1(X, j_\ast \cF_\xi)$, we will use the Euler characteristic formula. To this end, we introduce some notation. Let $x \in S_0(k)$, and let $K_x$ denote the completion of $k(X)$ at the point $x$, and let $I_x \subset \pi_1(U, \overline{\eta})$ denote a choice of inertia group at $x$. The restriction $\rho|_{I_x}$ gives rise to a Weil--Deligne representation $(r_x, N_x)$, where $r_x$ factors through a finite quotient $G_x = \Gal(L_x/K_x)$ of $I_x$ and $N_x$ is a nilpotent element of $\Lie G \subset \Lie \GL(\cF_{\overline{\eta}})$. We also define $H_x = r_x^{-1}(Z_G)$ and $K_x' = L_x^{H_x}$. Artin's representation $a_{G_x}$ is defined, and we obtain:
\begin{proposition}\label{prop_raynaud}
Suppose that $\xi \neq 1$. Then we have \[ \dim_{\overline{\bbQ}_l} H^1(X, j_\ast \cF_\xi) = -\chi(X) \dim_{\overline{\bbQ}_l} \xi + \sum_{x \in S_0(k)} \left( \langle a_{G_x}, r_x \rangle_{G_x} + \dim_{\overline{\bbQ}_l} r_x^{I_x} - \dim_{\overline{\bbQ}_l} r_x^{I_x} \cap \ker N_x \right). \]
\end{proposition}
\begin{proof}
The Euler characteristic formula \cite[Th\'eor\`eme 1]{Ray95} implies an equality
\[ \chi(X, j_\ast \cF_\xi) = \chi(X) \dim_{\overline{\bbQ}_l} \xi - \sum_{x \in S_0(k)} \left( \langle a_{G_x}, r_x \rangle_{G_x} + \dim_{\overline{\bbQ}_l} r_x^{I_x} - \dim_{\overline{\bbQ}_l} r_x^{I_x} \cap \ker N_x \right). \]
Our result follows from this on noting that the groups $H^0(X, j_\ast \cF_\xi)$ and $H^2(X, j_\ast \cF_\xi)$ are zero (because $\xi$ is non-trivial).
\end{proof}
The Lefschetz trace formula has a simple form when applied to the group $H^1(X, j_\ast \cF_\xi)$:
\begin{proposition}\label{prop_lefschetz}
Suppose that $\xi \neq 1$. Then for each $n \geq 1$, we have
\[ \tr ( \Frob_{q^n} \mid H^1(X, j_\ast \cF_\xi) ) = - \sum_{x \in X_0(\bbF_{q^n})} \tr ( \Frob_{q^n, x} \mid \cF_{\xi, 0, \overline{\eta}}^{I_x} ). \]
\end{proposition}
\begin{proof}
After extending scalars to $\bbF_{q^n}$, we can assume that $n = 1$. Then the usual Lefschetz trace formula gives
\[ \sum_{i=0}^2 (-1)^i \tr( \Frob_q \mid H^i_c(U, \cF_\xi) ) = \sum_{x \in U_0(\bbF_q)} \tr ( \Frob_{q, x} \mid \cF_{\xi, 0, \overline{\eta}} ). \]
The required formula now follows on taking into account the vanishing of the groups $H^0_c(U, \cF_\xi)$ and $H^2_c(U, \cF_\xi)$ and the short exact sequence of sheaves on $X_0$
\[ \xymatrix@1{ 0 \ar[r] & j_! \cF_\xi \ar[r] & j_\ast \cF_\xi \ar[r] & j_\ast \cF_\xi / j_! \cF_\xi \ar[r] & 0.} \]
\end{proof}
We now come to our main result. We suppose given a sequence $\xi_1, \xi_2, \dots$ of irreducible representations of $G$. We suppose that the central characters $\omega = \omega_{\xi_i}$ are independent of $i$. For each $x \in S_0(k)$, we define an invariant $\epsilon_x$ by the formula
\begin{equation} [K'_x : K_x] \epsilon_x =  v_{K_x} ( \mathfrak{d}_{K'_x/K_x} ) + \mathfrak{f}(\omega|_{H_x}) + \left\{ \begin{array}{ll} 0 & N_x = 0 \text{ or }N_x \neq 0 \text{ and }\omega|_{H_x} \neq 1. \\ 1 & N_x \neq 0 \text{ and }\omega|_{H_x} = 1. \end{array}\right. 
\end{equation}
(Here $v_x : K_x^\times \to \bbZ$ is the normalized valuation, $\mathfrak{d}_{K'_x/K_x}$ is the relative discriminant, and $\mathfrak{f}$ is the conductor; see \cite[Deuxi\`eme partie]{Ser68}. Thus, for example, $v_{K_x} ( \mathfrak{d}_{K'_x/K_x} ) = [ K'_x : K_x ] -1$ if the extension $K'_x/K_x$ is tamely ramified.)
\begin{theorem}\label{thm_equidistribution_general_case}
Suppose that $G \neq \{ 1 \}$, and let $\xi_1, \xi_2, \dots$ be a sequence of irreducible representations of $G$ such that $\xi_i \to \infty$ as $i \to \infty$ and $\omega = \omega_{\xi_i}$ is independent of $i$. Suppose that $-\chi(X) + \sum_{x \in S_0(k)} \epsilon_x \neq 0$. Then for each $n \geq 1$, the limit
\begin{equation}\label{eqn_existence_of_limit} \lim_{i \to \infty} \frac{q^{-n/2}}{h^1(X, j_\ast \cF_{\xi_i})} \iota \tr( \Frob_{q^n} \mid H^1(X, j_\ast \cF_{\xi_i}) ) 
\end{equation}
exists, and equals
\begin{multline}\label{eqn_value_of_limit} \omega_n = q^{-n/2} \left[ \sum_{\substack{x \in U_0(\bbF_{q^n}) \\ \rho_0(\Frob_{q^n, x}) \in Z_G}} \iota\omega(\Frob_{q^n, x}) + \sum_{x \in S_0(\bbF_{q^n})} \left\{ \begin{array}{ll} 0 & N_x \neq 0 \\ \frac{1}{|G_x|} \sum_{\substack{ \sigma \in G_x \\ \rho_0(\phi_{q^n, x} \sigma) \in Z_{G_0}}} \omega(\phi_{q^n, x} \sigma) & N_x = 0. \end{array} \right\} \right] \\
\times \left[ -\chi(X) + \sum_{x \in S_0(k)} \epsilon_x \right]^{-1},
\end{multline}
where for $x \in S_0(\bbF_{q^n})$, $\phi_{q^n, x} \in \pi_1(U_0, \overline{\eta})$ is a lift of $\Frob_{q^n, x}$. Let $\mathbf{x}_{\xi_i}$ denote the set of eigenvalues of $\Frob_q \otimes q^{-1/2}$ on $H^1(X, j_\ast \cF_\xi) \otimes_{\overline{\bbQ}_l, \iota} \bbC$, with multiplicity. Then there exists a unique measure $\nu$ on $S^1$ such that the sets $\mathbf{x}_{\xi_i} \subset S^1$ become $\nu$-equidistributed as $i \to \infty$.
\end{theorem}
The condition $-\chi(X) + \sum_{x \in S_0(k)} \epsilon_x \neq 0$ is very often satisfied; since $\epsilon_x \geq 0$, it in particular holds whenever $X$ has genus at least 2. 

Before giving the proof of Theorem \ref{thm_equidistribution_general_case}, we state another version under simplifying hypotheses:
\begin{corollary}\label{cor_equidistribution_special_case}
Suppose that $G \neq \{ 1 \}$, and let $\xi_1, \xi_2, \dots$ be a sequence of irreducible representations of $G$ such that $\xi_i \to \infty$ as $i \to \infty$ and $\omega = \omega_{\xi_i}$ is independent of $i$. Suppose moreover that for each $x \in S_0(k)$, $\cF$ is tamely ramified at $x$ and $N_x \neq 0$. For each $n \geq 1$, the limit
\begin{equation}\lim_{i \to \infty} \frac{q^{-n/2}}{h^1(X, j_\ast \cF_{\xi_i})} \iota \tr( \Frob_{q^n} \mid H^1(X, j_\ast \cF_{\xi_i}) ) 
\end{equation}
exists, and equals
\begin{equation}
 \frac{-q^{-n/2}}{\chi(U)} \sum_{\substack{x \in U_0(\bbF_{q^n}) \\ \rho_0(\Frob_{q^n, x}) \in Z_G}} \iota\omega(\Frob_{q^n, x}).
\end{equation}
\end{corollary}
We observe that the condition $\xi \neq 1$ implies that the group $G$ is non-abelian. This shows that in the situation of Corollary \ref{cor_equidistribution_special_case}, the Euler characteristic $\chi(U)$ is always non-zero.
\begin{proof}
This follows from Theorem \ref{thm_equidistribution_general_case} on noting that $\epsilon_x = 1$ for each $x \in S_0(k)$. Indeed, the hypothesis of tame ramification implies that $v_{K_x}(\mathfrak{d}_{K'_x/K_x}) = [K'_x : K_x] - 1$ and that $\mathfrak{f}(\omega|_{H_x}) = 1 - \langle \omega|_{H_x}, 1 \rangle_{H_x}$.
\end{proof}
\begin{proof}[Proof of Theorem \ref{thm_equidistribution_general_case}]
The last sentence of the theorem follows from the rest by Proposition \ref{prop_existence_of_limiting_measure}. It therefore suffices to calculate the limit (\ref{eqn_existence_of_limit}). Replacing $q$ by $q^n$, we can assume that $n = 1$. By Proposition \ref{prop_lefschetz}, we have for any non-trivial $\xi$:
\begin{multline}\label{eqn_big_equation} \frac{q^{-1/2}}{h^1(X, j_\ast \cF_{\xi})} \iota \tr( \Frob_{q} \mid H^1(X, j_\ast \cF_{\xi})) \\ = q^{-1/2} \cdot \frac{ \dim_{\overline{\bbQ}_l} \xi}{h^1(X, j_\ast \cF_{\xi})} \cdot \iota \left[ \sum_{x \in U_0(\bbF_{q})} \frac{\tr (\xi \circ \rho_0)(\Frob_{q, x})}{\dim_{\overline{\bbQ}_l} \xi} + \sum_{x \in S_0(\bbF_q)} \frac{\tr (\xi \circ \rho_0)^{I_x}(\Frob_{q, x})}{\dim_{\overline{\bbQ}_l} \xi} \right]. 
\end{multline}
We first observe that the quotient $\dim_{\overline{\bbQ}_l} \xi / h^1(X, j_\ast \cF_{\xi})$ tends to a non-zero limit as $\xi \to \infty$. Indeed, by Proposition \ref{prop_raynaud}, we have
\[ \frac{h^1(X, j_\ast \cF_{\xi})}{\dim_{\overline{\bbQ}_l} \xi} = - \chi(X) + \sum_{x \in S_0(\bbF_q)} \left( \frac{ \langle a_{G_x}, \xi \circ r_x \rangle_{G_x} }{ \dim_{\overline{\bbQ}_l} \xi } + \frac{\dim_{\overline{\bbQ}_l} (\xi \circ r_x)^{G_x}}{\dim_{\overline{\bbQ}_l} \xi} - \frac{\dim_{\overline{\bbQ}_l} ((\xi \circ r_x)^{G_x} \cap \ker N_x)}{\dim_{\overline{\bbQ}_l} \xi} \right).  \]
For $x \in S_0(\bbF_q)$, we have
\[ \frac{\langle a_{G_x}, \xi \circ r_x \rangle_{G_x}}{\dim_{\overline{\bbQ}_l} \xi} = \frac{1}{|G_x|} \sum_{g \in G_x} \frac{ a_{G_x}(g) \cdot \tr \xi( r_x (g) )}{\dim_{\overline{\bbQ}_l} \xi}, \]
hence (calculating the limit using Proposition \ref{prop_chenevier_clozel})
\[ \lim_{\xi \to \infty} \frac{\langle a_{G_x}, \xi \circ r_x \rangle_{G_x}}{\dim_{\overline{\bbQ}_l} \xi} = \frac{1}{|G_x|} \sum_{g \in H_x} a_{G_x}(g) \iota\omega(g) = \frac{1}{[K_x' : K_x]} \langle a_{G_x}|_{H_x}, \omega\rangle_{H_x}. \]
Using the formula $a_{G_x}|_{H_x} = v_{K_x}(\mathfrak{d}_{K'_x/K_x}) \cdot r_{H_x} + a_{H_x}$ ($r_{H_x}$ the regular representation of $H_x$), we finally obtain
\[ \lim_{\xi \to \infty} \frac{\langle a_{G_x}, \xi \circ r_x \rangle_{G_x} }{\dim_{\overline{\bbQ}_l} \xi} = \frac{v_{K_x}(\mathfrak{d}_{K'_x/K_x}) + \mathfrak{f}(\omega|_{H_x})}{[K_x' : K_x]}. \]
A similar argument shows that
\[ \lim_{\xi \to \infty} \frac{ \dim_{\overline{\bbQ}_l} (\xi \circ r_x)^{I_x} }{ \dim_{\overline{\bbQ}_l}\xi} = \frac{ \langle \omega, 1 \rangle_{H_x} }{[K_x' : K_x]}. \]
If $N_x \neq 0$, then Proposition \ref{prop_jacobson_morozov} implies that 
\[ \lim_{\xi \to \infty} \frac{\dim_{\overline{\bbQ}_l} ((\xi \circ r_x)^{I_x} \cap \ker d \xi(N_x))}{\dim_{\overline{\bbQ}_l} \xi} = 0. \]
In any case, we obtain 
\begin{equation}\label{eqn_limit_equals_epsilon} \lim_{\xi \to \infty} \frac{h^1(X, j_\ast \cF_{\xi})}{\dim_{\overline{\bbQ}_l} \xi} = - \chi(X) + \sum_{x \in S_0(k)} \epsilon_x. 
\end{equation}
The value of (\ref{eqn_limit_equals_epsilon}) is non-zero, by assumption. We calculate the limiting values of the other terms in (\ref{eqn_big_equation}). If $x \in U_0(\bbF_q)$, then we have
\begin{equation}\label{eqn_unramified_limit} \lim_{\xi \to \infty} \frac{ \iota\tr (\xi \circ \rho_0)(\Frob_{q, x}) }{\dim_{\overline{\bbQ}_l} \xi} = \left\{ \begin{array}{cc} 0 & \rho_0(\Frob_{q, x}) \not \in Z_{G} \\
\iota\omega(\Frob_{q, x}) & \rho_0(\Frob_{q, x}) \in Z_{G}, \end{array}\right. 
\end{equation}
by Proposition \ref{prop_chenevier_clozel}. If $x \in S_0(\bbF_q)$, let $\phi_{q, x} \in \pi_1(U_0, \overline{\eta})$ be a choice of Frobenius lift at $x$. In order to calculate the limit
\begin{equation}\label{eqn_boundary_limit} \lim_{\xi \to \infty} \frac{\iota\tr (\xi \circ \rho_0)^{I_x}(\Frob_{q, x})}{\dim_{\overline{\bbQ}_l} \xi}, 
\end{equation}
we split into cases according to whether or not $N_x = 0$. If $N_x \neq 0$, then we have $| \iota \tr (\xi \circ \rho_0)^{I_x}(\Frob_{q, x}) | \leq \dim_{\overline{\bbQ}_l} \ker d \xi(N_x)$, and it follows from Proposition \ref{prop_jacobson_morozov} that the limit \ref{eqn_boundary_limit} is 0. If $N_x = 0$, then we calculate
\[ \tr (\xi \circ \rho_0)^{I_x}(\Frob_{q, x}) = \frac{1}{|G_x|} \sum_{\sigma \in G_x} \tr (\xi \circ r_x)(\phi_{q, x} \sigma), \]
hence
\[ \lim_{\xi \to \infty} \frac{  \iota\tr (\xi \circ \rho_0)(\Frob_{q, x}) }{\dim_{\overline{\bbQ}_l} \xi} =  \frac{1}{|G_x|} \sum_{\substack{\sigma \in G_x \\ \rho_0(\phi_{q, x} \sigma) \in Z_{G}}} \iota\omega(\phi_{q, x} \sigma). \]
Combining equations (\ref{eqn_limit_equals_epsilon}) and  (\ref{eqn_unramified_limit}) now shows that the limit (\ref{eqn_big_equation}) exists, and is equal to the claimed value $\omega_n$. This completes the proof of the theorem.
\end{proof}

\section{Application to elliptic curves}\label{sec_examples}

Let $X_0$ be a smooth, geometrically connected, projective curve over $\bbF_q$. We observe that Theorem \ref{thm_equidistribution_general_case} applies to any elliptic curve over $\bbF_q(X_0)$ of non-constant $j$-invariant. More precisely, any elliptic curve $E$ over $\bbF_q(X_0)$ of non-constant $j$-invariant spreads out to an elliptic surface $\pi_0 : \cE_0 \to U_0$, for some Zariski open subset $U_0 \subset X_0$, and gives rise to a dominant morphism $X_0 \to \bbP^1_{\bbF_q}$ to the $j$-line. If $l \neq p$, then the rank 2 lisse $\overline{\bbQ}_l$-sheaf $\cF = R^1 \pi_\ast \overline{\bbQ}_l$ has geometric monodromy group $G = \SL_2$ (see \cite[Lemme 3.5.5]{Del80}). Choosing a square root $\epsilon^{1/2}$ of the cyclotomic character, we define $\cF_0 = R^1 \pi_{0, \ast} \overline{\bbQ}_l \otimes \epsilon^{1/2}$; then we have $G_0 = G = \SL_2$. 

If $x \in U_0(\bbF_{q^n})$ is a point such that $\rho_0(\Frob_{q^n, x}) \in Z_G$ is scalar, then the fiber $\cE_{0, x}$ is a supersingular elliptic curve over $\bbF_{q^n}$; the number of such points is uniformly bounded in $n$ (since there are only finitely many supersingular $j$-invariants), implying that the limiting values $\omega_n$ in Theorem \ref{thm_equidistribution_general_case} are $O(q^{-n/2})$. Consequently, the limiting measure $\nu$ is described by integration against a real analytic function (cf. the discussion in \S \ref{sec_equidistribution_on_the_circle}).

As an example of this, let $p$ be an odd prime and consider the Legendre elliptic curve
\[ E : y^2 = x(x-1)(x-\lambda) \]
over the field $\bbF_p(\lambda)$, forgetting its interpretation as a modular curve. This curve has multiplicative reduction at the points $\lambda = 0, 1$ and additive, potentially multiplicative reduction at the point $\lambda = \infty$. We are therefore in the situation of Corollary \ref{cor_equidistribution_special_case}, with $U_0 = \bbP^1_{\bbF_p} - \{ 0, 1, \infty \}$. As is well-known (see \cite[Theorem 4.1]{Sil09}), there are exactly $(p-1)/2$ values of $\lambda$ for which the curve $E_\lambda$ is supersingular, and these all lie in $\bbF_{p^2}$. If $\alpha$ is one of these values, then $\rho_0(\Frob_{{p^n}, \alpha})$ is scalar if and only if $n$ is even.

If $i \geq 1$, let $\xi_i$ denote the representation $\Sym^{2i}$ of $\SL_2$, which has trivial central character. Applying Corollary \ref{cor_equidistribution_special_case}, we find that for each $n \geq 1$, the limit
\[ \lim_{i \to \infty} \frac{p^{-n/2}}{h^1(X, j_\ast \cF_{\xi_i})} \iota \tr( \Frob_{p^n} \mid H^1(X, j_\ast \cF_{\xi_i}) ) \]
exists, and is equal to $(p-1)/2$ if $n$ is even, and 0 otherwise. As $i \to \infty$, the eigenvalues of $ \Frob_p \otimes p^{-1/2}$ on $H^1(X, j_\ast \cF_{\xi_i}) \otimes_{\overline{\bbQ}_l, \iota} \bbC$ therefore become equidistributed with respect to the measure 
\[ \nu(f) = \int_{S^1} f F d\mu, \text{ } F(z) = 1 + (p-1)\sum_{n \geq 1} p^{-n} \cos(2 n \theta). \]
This is exactly the measure, related to Plancherel measure for $\PGL_2(\bbQ_p)$, discussed at the end of \S \ref{sec_equidistribution_on_the_circle}. 

\section{Application to Kloosterman sheaves}\label{sec_examples_2}

We now discuss an example coming from the theory of Kloosterman sheaves. Let $p$ be a prime, and let $q$ be a power of $p$. Let $X_0 = \bbP^1_{\bbF_q}$, let $N \geq 3$ be an odd integer, let $\psi : \bbF_q \to \overline{\bbQ}_l^\times$ be a non-trivial character, and let $\cF_0$ be the sheaf on $\bbA^1_{\bbF_q} - \{ 0 \}$ denoted $R^{N-1} \pi_! \cF( \psi \sigma)$ in \cite[Th\'eor\`eme 7.8]{Del77}. Then $\cF_0$ is lisse of rank $N$, and if $a \in \bbA^1_{\bbF_q}(\bbF_q) - \{ 0 \}$, then we have
\[ \tr( \Frob_{q, a} \mid \cF_{0, \overline{\eta}} ) = (-1)^{N-1} \sum_{\substack{x_1, \dots, x_N \in \bbF_q \\ x_1 \cdots x_N = a}} \psi(x_1 + \dots + x_N), \]
a generalized Kloosterman sum. 

Now suppose that $p \geq N + 2$. By a theorem of Katz \cite{Kat88}, the geometric monodromy group is $G = \SL_N$, and after twisting by a character we can force the arithmetic monodromy group to be equal to $G$ as well (cf. \cite[Lemma 3.1]{Kat13}). It therefore makes sense to apply Theorem \ref{thm_equidistribution_general_case} to the sheaves $\cF_{0, \xi_i}$ for any sequence of representations $\xi_i$ of $\GL_N$ with $\xi_i \to \infty$ as $i \to \infty$ and with trivial central characters. In fact, we will prove:
\begin{theorem}
With assumptions as above, let $\xi_1, \xi_2, \dots$ be a sequence of representations of $\GL_N$ with trivial central character such that $\xi_i \to \infty$ as $i \to \infty$. Then there is a measure $\nu$ on $S^1$ such that the eigenvalues of $\Frob_q \otimes q^{-1/2}$ on the space $H^1(X, j_\ast \cF_{\xi_i}) \otimes_{\overline{\bbQ}_l, \iota} \bbC$ become $\nu$-equidistributed as $i \to \infty$.

The measure $\nu$ can be written explicitly as follows. There is an integer $b \geq 1$ such that $\nu(f) = \int_{S^1} f F \, d\mu$, where $F(z) = 1 + 2 \sum_{n \geq 1} \omega_n \cos (n \theta)$ is the analytic function on $S^1$ given by the formula
\[ \omega_n = \left\{\begin{array}{ll} 0 & b \nmid n \\ q^{-n/2}( p^a -2)^{-1} & b|n, \end{array}\right. \]
where $a = [ \bbF_p (\zeta_N) : \bbF_p]$, $\zeta_N \in k$ a primitive $N^\text{th}$-root of unity.
\end{theorem}
\begin{proof}
We will apply Theorem \ref{thm_equidistribution_general_case} with $U_0 = \bbG_{m, \bbF_q}$. By a theorem of Sperber (\cite[Theorem 2.35]{Spe80}), for any $a \in U_0(\bbF_{q^n})$, the element $\Frob_{q^n, a} \in \GL(\cF_{0, \overline{\eta}})$ is regular semi-simple. In particular, it is non-central, so the only contribution to $\omega_n$ in the formula of Theorem \ref{thm_equidistribution_general_case} comes from the boundary $S_0 = \{ 0 , \infty \}$. To apply the theorem, we must first check that the value $-\chi(X) + \epsilon_0 + \epsilon_\infty$ is non-zero. We have $\chi(X) = 2$ and $\epsilon_0 = 1$, since the inertia group at $0$ acts unipotently, with the logarithm of tame inertia having a single Jordan block, by \cite[Th\'eor\`eme 7.8]{Del77}.

The key point will be the determination of the local structure of the sheaf $\cF_0$ at the point $\infty$, where it is wildly ramified. To this end, we introduce some extra notation. Let $D \subset \GL(\cF_{0, \overline{\eta}})$ denote the image of a decomposition group at infinity and $I \subset D$ the image of the inertia group. Let $\overline{D}$ denote the image of $D$ in $\PGL(\cF_{0, \overline{\eta}})$, and $\overline{I}$ the image of $I$. We then have the following facts:
\begin{itemize}
\item The group $\overline{D}$ is finite, and the Swan conductor of $D$ in its natural representation is equal to 1.
\item Let $\overline{I}_0 \supset \overline{I}_1 \supset \dots$ denote the lower ramification filtration on $\overline{I}$. Then $\overline{I}_2 = \{ 1 \}$, $\overline{I}_1$ is contained in a unique maximal torus $T$ of $\PGL(\cF_{0, \overline{\eta}})$, and $\overline{I}_1 = T \cap \overline{I}$. Moreover, $\overline{I}_0 / \overline{I}_1$ is generated by a Coxeter element (i.e.\ an $N$-cycle) in $\overline{I}_0 / \overline{I}_1 \hookrightarrow N_{\PGL}(T) / T \cong S_N$, and $\overline{I}_1$ has order $p^a$.
\end{itemize}
The finiteness of $\overline{D}$ follows from \cite[Lemma 1.11]{Kat88}. The computation of the Swan conductor is due to Deligne \cite[Th\'eor\`eme 7.8]{Del77}. The calculation of $\overline{I}$ with its ramification filtration has been carried out by Yun \cite[\S 2.3.4]{Yun13} using results of Gross--Reeder \cite{Gro10}. 

Using the formula for the valuation of the relative discriminant in terms of ramification groups \cite[Ch. IV, \S 1, Proposition 4]{Ser68}, we calculate $ \epsilon_\infty = (| \overline{I}_0 | - 1 + | \overline{I}_1 | - 1)/p^a N = 1 + (p^a - 2)/p^a N$, and so $-\chi(X) + \epsilon_0 + \epsilon_\infty = (p^a - 2)/p^a N > 0$. Similarly, let $\phi_{q, \infty} \in \pi_1(U_0, \overline{\eta})$ be a Frobenius lift at infinity and let $\overline{F}$ denote the image of $\phi_{q, \infty}$ in $\overline{D}$. Let $b$ denote the smallest value of $n$ such that $\overline{F}^n \in \overline{I}$; equivalently, $b$ is the order of the cyclic group $\overline{D} / \overline{I}$. We obtain for each $n \geq 1$:
\[ \omega_n = q^{-n/2} \frac{p^a N}{p^a - 2} \frac{\# \{ \sigma \in \overline{I} \mid \sigma \overline{F} = 1 \}}{\# \overline{I}} = \left\{ \begin{array}{ll} 0 & b \nmid n \\
\frac{q^{-n/2}}{p^a - 2} & b | n. \end{array} \right. \]
\end{proof}
\begin{remark}
\begin{enumerate}
\item The determination of the integer $b$ is equivalent to the determination of the full image $\overline{D}$ of the decomposition group at infinity, as opposed to just the image $\overline{I}$ of the inertia group which is used in the proof. We do not attempt this here. However, let us observe that a necessary condition to have $b = 1$ is that $q \equiv 1 \text{ mod }N$, and that we can always force $b = 1$ by extending the base field $\bbF_q$ of scalars (in effect, forcing $\overline{D} = \overline{I}$).
\item The fact that Frobenius acts on stalks in $U_0$ through regular semi-simple elements means that the appeal to Proposition \ref{prop_chenevier_clozel} in the proof of Theorem \ref{thm_equidistribution_general_case} can be replaced in this instance with an appeal to the Weyl character formula. This allows one to give a better bound for the convergence of the limits $\omega_n$, and to replace the condition $\xi_i \to \infty$ with the weaker assertion that $\dim_{\overline{\bbQ}_l} \xi_i \to \infty$ as $i \to \infty$ (central characters still being fixed).
\end{enumerate}
\end{remark}

\bibliographystyle{alpha}
\bibliography{frobenius}

\end{document}